\documentclass[11pt,twoside]{amsart}

\def\authornik{}

\usepackage[cp1251]{inputenc}
\usepackage[T2A]{fontenc}
\usepackage[ukrainian,russian,english]{babel}
\usepackage[dvips]{graphicx,graphics}
\usepackage{latexsym,amsfonts,amsmath,amscd}
\usepackage{amssymb,mathtools}
\usepackage{fancyhdr}
\usepackage{ifthen}
\usepackage{afterpage}
\usepackage{footnpag}

\usepackage[total={110mm,170mm},includehead,includefoot,a4paper,asymmetric]{geometry}

\newtheorem{theorem}{\THEOREM}
\newtheorem{lemma}{\LEMMA}
\newtheorem{prop}{\PROPOSITION}
\newtheorem{cor}{\COROLLARY}
\newtheorem{defn}{\DEFINITION}
\newtheorem{rem}{\REMARK}
\newtheorem{example}{\EXAMPLE}
\newtheorem{problem}{\PROBLEM}
\newtheorem{question}{\QUESTION}

\makeatletter
\@addtoreset{equation}{section}
\makeatother

\renewcommand{\emptyset}{\varnothing}

\newcommand{\SetTheorems}{%
	\SetLangDep{\THEOREM}{Теорема}{Теорема}{Theorem}%
	\SetLangDep{\LEMMA}{Лема}{Лемма}{Lemma}%
	\SetLangDep{\PROPOSITION}{Твердження}{Предложение}{Proposition}%
	\SetLangDep{\COROLLARY}{Наслідок}{Следствие}{Corollary}%
	\SetLangDep{\DEFINITION}{Означення}{Определение}{Definition}%
	\SetLangDep{\REMARK}{Зауваження}{Замечание}{Remark}%
	\SetLangDep{\EXAMPLE}{Приклад}{Пример}{Example}%
	\SetLangDep{\DESCRIPTION}{Опис}{Описание}{Description}%
	\SetLangDep{\PROBLEM}{Приклад}{Пример}{Problem}%
	\SetLangDep{\QUESTION}{Приклад}{Пример}{Question}%
}
\newcommand{\ResetCounters}{%
\setcounter{theorem}{0}%
\setcounter{lemma}{0}%
\setcounter{prop}{0}%
\setcounter{cor}{0}%
\setcounter{defn}{0}%
\setcounter{rem}{0}%
}

\newcounter{article_num}
\newcommand\SET[2]{\def#1{#2}}
\newcommand\SetLangDep[4]{%
	\def#1{
	\ifthenelse{\equal{\LANGUAGE}{ukrainian}}{#2}{%
	\ifthenelse{\equal{\LANGUAGE}{russian}}{#3}{%
	\ifthenelse{\equal{\LANGUAGE}{english}}{#4}{}%
	}%
	}%
	}%
}

\SET{\Year}{2011}
\SET{\Volume}{2}
\SET{\Number}{1}

\newcommand\UDK[1]{\def\MYUDK{#1}}
\newcommand\TITLE[1]{\def\MYTITLE{#1}}
\newcommand\AUTHORa[1]{\def\MYAUTHORA{#1}}
\newcommand\AUTHORb[1]{\def\MYAUTHORB{#1}}
\newcommand\AUTHORc[1]{\def\MYAUTHORC{#1}}
\newcommand\AUTHORd[1]{\def\MYAUTHORD{#1}}
\newcommand\ADDRESSa[1]{\def\MYADDRESSA{#1}}
\newcommand\ADDRESSb[1]{\def\MYADDRESSB{#1}}
\newcommand\ADDRESSc[1]{\def\MYADDRESSC{#1}}
\newcommand\ADDRESSd[1]{\def\MYADDRESSD{#1}}
\newcommand\EMAILa[1]{\def\MYEMAILA{#1}}
\newcommand\EMAILb[1]{\def\MYEMAILB{#1}}
\newcommand\EMAILc[1]{\def\MYEMAILC{#1}}
\newcommand\EMAILd[1]{\def\MYEMAILD{#1}}
\newcommand\KEYWORDS[1]{\def\MYKEYWORDS{#1}}
\newcommand\abstractukr[1]{\def\myabstractukr{#1}}
\newcommand\abstractrus[1]{\def\myabstractrus{#1}}
\newcommand\abstracteng[1]{\def\myabstracteng{#1}}
\newcommand\THANKS[1]{\def\MYTHANKS{#1}}
\newcommand\HEADERS[2]{\def\MYSHORTTITLE{#1}  \def\MYSHORTAUTHORS{#2}}

\newcommand\CLEANALL{
\UDK{\empty}
\TITLE{\empty}
\AUTHORa{\empty}
\AUTHORb{\empty}
\AUTHORc{\empty}
\AUTHORd{\empty}
\ADDRESSa{\empty}
\ADDRESSb{\empty}
\ADDRESSc{\empty}
\ADDRESSd{\empty}
\EMAILa{\empty}
\EMAILb{\empty}
\EMAILc{\empty}
\EMAILd{\empty}
\KEYWORDS{\empty}
\abstractukr{\empty}
\abstractrus{\empty}
\abstracteng{\empty}
\SET{\myfirstpage}{\empty}
\SET{\mylastpage}{\empty}
\THANKS{\empty}
\HEADERS{\empty}{\empty}
}

\newcommand\NEWARTICLE[1]{
\CLEANALL
\addtocounter{article_num}{1}
\SET{\LANGUAGE}{#1}
\selectlanguage{#1}
\SetLangDep{\KEYWORDSNAME}{Ключові слова}{Ключевые слова}{Keywords}
\SET{\myfirstpage}{page-first-\authornik}
\SET{\mylastpage}{page-last-\authornik}
\SetTheorems
\ResetCounters
}

\newcommand\authorslist{{\bf \copyright~~\MYAUTHORA %
\ifthenelse{\equal{\MYAUTHORB}{}}%
           {%
              \ifthenelse{\equal{\MYAUTHORC}{}}%
                         {%
                                            \ifthenelse{\equal{\MYAUTHORD}{}}%
                                            {}%
                                            {, \MYAUTHORD}%
                         }%
                         {, \MYAUTHORC}%
           }%
           {, \MYAUTHORB},  \Year%
}}

\fancypagestyle{initpage}{%
	\fancyhf{} 
	\fancyfoot[R]{\authorslist} 
	\addtolength{\footskip}{14pt} 

}

\fancypagestyle{myplain}{
    \fancyhf{}  
    \setlength{\headheight}{14pt}
    \fancyhead[LO]{\sl \MYSHORTTITLE} \fancyhead[RO]{\sl \thepage}
    \fancyhead[LE]{\sl \thepage}    \fancyhead[RE]{\sl \MYSHORTAUTHORS}
}

\newcommand\makeheaders{\pagestyle{myplain}\thispagestyle{initpage}}

\newcommand\makeabstract[1]{
\begin{flushright}
 \begin{minipage}{0.95\textwidth}
 \footnotesize #1
 \end{minipage}
\medskip
\end{flushright}
}

\newcommand\MAKETITLE{%
\makeheaders
\vspace*{-2.2\headheight}
{\selectlanguage{ukrainian}\begin{flushright}
\end{flushright}}

\begin{flushleft}
\vspace*{0.5cm}
\ifthenelse {\equal{\MYUDK}{}} {} {\selectlanguage{ukrainian} \MYUDK  \\ \medskip}
{\it\bfseries \large \MYAUTHORA \\ \smallskip }
\ifthenelse{\equal{\MYADDRESSA}{}}{}{ {\it \MYADDRESSA}\\ }
\ifthenelse{\equal{\MYEMAILA}{}}{} { {\it E-mail:} \texttt{\MYEMAILA}  \\ \medskip\smallskip }  %
\ifthenelse{\equal{\MYAUTHORB}{}}{}{ %
  {\it\bfseries \large\MYAUTHORB} \\  \smallskip
  \ifthenelse{\equal{\MYADDRESSB}{}}{}{ {\it \MYADDRESSB}\\ }
  \ifthenelse{\equal{\MYEMAILB}{}}{} { {\it E-mail:} \texttt{\MYEMAILB}  \\ } \medskip\smallskip %
}
\ifthenelse{\equal{\MYAUTHORC}{}}{}{ %
  {\it\bfseries \large\MYAUTHORC} \\  \smallskip
  \ifthenelse{\equal{\MYADDRESSC}{}}{}{ {\it \MYADDRESSC}\\ }
  \ifthenelse{\equal{\MYEMAILC}{}}{} { {\it E-mail:} \texttt{\MYEMAILC}  \\ } \medskip\smallskip %
}
\ifthenelse{\equal{\MYAUTHORD}{}}{}{ %
  {\it\bfseries \large\MYAUTHORD} \\  \smallskip
  \ifthenelse{\equal{\MYADDRESSD}{}}{}{ {\it \MYADDRESSD}\\ }
  \ifthenelse{\equal{\MYEMAILD}{}}{} { {\it E-mail:} \texttt{\MYEMAILD}  \\ } \medskip\smallskip %
}
{\LARGE\bf\MYTITLE \ifthenelse{\equal{\MYTHANKS}{}}{}{\footnote{\MYTHANKS}}} \\ \medskip
\end{flushleft}
\ifthenelse{ \equal{\myabstractukr}{} }{} {{\selectlanguage{ukrainian}\makeabstract{\myabstractukr}}}
\ifthenelse{ \equal{\myabstractrus}{} }{} {{\selectlanguage{russian}\makeabstract{\myabstractrus}}}
\ifthenelse{ \equal{\myabstracteng}{} }{} {{\selectlanguage{english}\makeabstract{\myabstracteng}}}
\ifthenelse{ \equal{\MYKEYWORDS}{} }{} {{\footnotesize{\bf \KEYWORDSNAME}: {\it \MYKEYWORDS}}}
\medskip\setcounter{section}{0}
}

\DeclareMathAlphabet{\mathbi}{OML}{cmm}{b}{it}

\usepackage{amssymb}

\newcommand{\R}{{\mathbb R}}
\newcommand{\reals}{{\mathbb R}}
\newcommand{\N}{{\mathbb N}}
\newcommand{\Q}{{\mathbb Q}}
\newcommand{\Z}{{\mathbb Z}}
\newcommand{\la}{{\langle}}
\newcommand{\ra}{{\rangle}}
\newcommand{\CSP}{\mbox{\em CSP}}
\newcommand{\eps}{\varepsilon}
\newcommand{\nbd}{neighbourhood }

\newcommand{\ttt}{{tight}}

\begin{document}

\NEWARTICLE{english}

\TITLE{Topics in uniform continuity}



\AUTHORa{Dikran Dikranjan}   
\ADDRESSa{Dipartimento di Matematica e Informatica,
Universit\`{a} di Udine,  Via delle Scienze  206, 33100 Udine, Italy}
\EMAILa{dikranja@dimi.uniud.it} 

\AUTHORb{Du\v san Repov\v s}   
\ADDRESSb{Faculty of Mathematics and Physics,  and Faculty of Education, University of Ljubljana, P.O.B. 2964, 1001 Ljubljana, Slovenia } 
\EMAILb{dusan.repovs@guest.arnes.si} 


\abstracteng{This paper collects results and open problems concerning
several classes of functions that generalize  uniform continuity in various ways,
including those metric spaces (generalizing Atsuji spaces) where all continuous functions have the property of being close to uniformly continuous. }




\HEADERS{Topics in uniform continuity}{D. Dikranjan and D. Repov\v s}

\label{\myfirstpage}
\MAKETITLE


\medskip

\centerline{\large Dedicated to the memory of Jan Pelant (1950-2005)}

\medskip

\section{Introduction}

The uniform continuity of maps between metric or uniform spaces 
determines a specific topic in general topology. By the end of the fifties and in the seventies the attention was concentrated on those spaces (called UC spaces, or Atsuji spaces) on which uniform continuity coincides with continuity. 
Of course, compact spaces are UC spaces, but there also exist non-compact UC spaces (e.g., the uniformly discrete ones). 

In this survey we consider several aspects of uniform continuity and its relationship with continuity. We start with a discussion 
of the possibility to capture uniform continuity by means of the so-called closure operators \cite{DT}. The most relevant and motivating example
of a closure operator is the usual Kuratowski closure $K$ in the category ${\bf Top}$ of topological spaces and continuous maps. It is well known that
one can describe the morphisms in ${\bf Top}$ (i.e., the continuous maps) in an equivalent way as the maps ``compatible" with the 
 Kuratowski closure (see  \S 2.2). In this setting appear the {\em uniformly approachable} and the {\em weakly uniformly approachable} (briefly, {\em UA} and {\em  WUA}, resp.) functions (see  Definition \ref{UA:Definition}).
Section 3 compares the properties UA and WUA with the property of the u.c. functions which have  
distant fibers in an appropriate sense. In Section 4 we consider studies those spaces $X$ on which every continuous function $X\to \R$ is
UA. This class contains the well-known  {\em Atsuji } spaces, where every continuous function is u.c.  [A1, A2, Be, BDC]. 

Section 5 deals with those metric spaces in which the uniform quasi-components of every closed subspace are closed. Every UA space is thin, but there exist complete
thin spaces that are not UA. The main result of this section is a separation property of the complete thin spaces. 

The last section is dedicated to ``additivity", which turns out to be quite a non-trivial question in the case of uniform continuity. More precisely, we discuss here 
 the {\em straight} spaces: these are the metric spaces on which a continuous function $X \to \R$ is uniformly continuous
whenever the restrictions $f\restriction_{F_1}$ and $f\restriction_{F_2}$ on each member of an arbitrary closed binary cover $X = F_1\cup F_2$ are uniformly continuous. 

At the end of the paper we collect most of open problems and question (although some of them can be found in the main text). 

We did not include in this paper several related issues. One of them, {\em magic sets}, is a topic that appeared in connection with UA functions, even though it has no apparent connection to uniform continuity. The reader can see [B2, BC1, BC2, BeDi1, BeDi2, CS] for more on this topic.

We are dedicating this survey to the memory of the outstanding topologist and our good friend Jan Pelant, who actively worked on this topic and contributed the most relevant results in this area. 

\subsection{Notation and terminology} 

A topological space $X$ is called {\em hereditarily disconnected} if all connected components of $X$ are trivial,
 while $X$ is called {\em totally disconnected}, if  all   quasi components of $X$ are trivial \cite{E}.
The closure of a subset $Y$ of a topological space $X$ will be denoted by  $cl(Y)$ or $\overline Y$. 
All topological spaces considered  in this paper are assumed to be Tychonoff.
A topological space is said to be {\em Baire}\/ if it  satisfies the Baire Category Theorem,
i.e., if every meager subset of $X$ has  empty interior. A {\em Cantor set}\/  is a nonvoid zero-dimensional compact
metrizable space with no isolated points, i.e., a homeomorphic copy of the
Cantor middle thirds set.  If $X$ is a topological space, we write $
C(X)$ for the family of all continuous real-valued functions on $X$, and $C_n(X)$ for the sets of continuous nowhere constant real-valued functions on $X$. (To say 
that a continuous function $f\colon X\to\R$ is nowhere constant is equivalent to saying that $f^{-1}(y)$ is nowhere dense for each $y\in\R$.) 

A metric space is said to be {\em uniformly locally connected} (shortly, {\em ULC}) [HY, 3-2], if for every $\varepsilon > 0$ there is $\delta > 0$ such that any two points at
distance $< \delta$ lie in a connected set of diameter $< \varepsilon$.  In other words, close  points can be connected by small  connected sets. For example, 
convex subsets of  $\R^m$ (or any Banach space) are uniformly locally connected.

\section{Uniform continuity vs  continuity}
 
\subsection{Global view on closure operators}

Closure operators can be introduced in a quite general context \cite{DT}. The prominent examples that inspired this general notion
were given by Isbell (in the category of semigroups, or more generaly, categories of universal algebras) and Salbany (in the category ${\bf Top}$ of topological spaces and
continuous maps). 
We briefly recall here the notion of a closure operator of ${\bf Top}$, following \cite{DT,DTW}.

A {\it closure operator} of ${\bf Top}$ is a family $C=(c_X)_{X\in {\bf Top}}$ of maps $ c_X:2^X \longrightarrow 2^X $  such that for every $X$ in ${\bf Top}$
\begin{itemize}
\item[(i)]  $ M\subseteq c_X(M)$ for all $M\in 2^X$;
\item[(ii)] $M\subseteq M'\in 2^X \Rightarrow c_X(M)\subseteq c_X(M')$; and 
\item[(iii)] $f(c_X(M))\subseteq c_Y(f(M))$ for all  $f:X\to Y$ in ${\bf Top}$ and $M\in 2^X$.
\end{itemize}

A prominent example is the Kuratowski closure operator $K$. Every continuous function satisfies the ``continuity" condition (iii) for every closure operator $C$.  
For a  closure operator $C$ of ${\bf Top}$   we say that the set map $f:X \to Y$ is $C$-{\it continuous}, if it satisfies (iii).
It is easy to see that a map $f:X \to Y$  is continuous if and only if it is $K$-{\it continuous}. In other words, the morphisms in ${\bf Top}$ can be detected by a closure operator
(as $K$-{\it continuous} maps). 

Analogously, a  closure operator of ${\bf Unif}$ can be defined as a family $C=(c_X)_{X\in {\bf Unif}}$ of maps $ c_X:2^X \longrightarrow 2^X $  such that for every $X$ in ${\bf Unif}$
items (i) and (ii) are satisfied, and 
\begin{itemize}
\item[(iii$_u$)] $f(c_X(M))\subseteq c_Y(f(M))$ for all  $f:X\to Y$ in ${\bf Unif}$ and $M\in 2^X$.
\end{itemize}
 
We say that $C$ is {\em additive} ({\em idempotent}) if the equality $c_X(M\cup N)=c_X(M)\cup c_X(N)$ (resp., $c_X(c_X(M))=c_x(M)$) always holds. 

The Kuratowski closure operator $K$ is a closure operator of ${\bf Unif}$. Analogously, for a  closure operator $C$ of ${\bf Unif}$  one can say that the set map $f:X \to Y$ is $C$-{\it continuous}, if it satisfies (iii$_u$); $f$ is said to be   {\it totally continuous} if it is $C$-continuous for every closure operator $C$ of ${\bf Unif}$. 

 In the category ${\bf Top}$ of topological spaces a map is continuous if and only if (iii) is holds for $C=K$.  Hence morphisms in
 ${\bf Top}$ are determined by the closure operator $K$.  Can the same be said of ${\bf Unif}$?  This question was answered in the negative in \cite{DP}. We briefly sketch the proof here. 

The spaces needed as tools are the uniformly discrete two-point space $D=\{0,1\}$, the one-point compactification  $\N _\infty $ of the naturals $\N$  equipped with its unique uniformity, and two uniformly close sequences, which are not topologically close
$$
X_0:=\{(n,1/n)\vert\ n\in{\N}\}\cup\{(n,-1/n)\vert\ n\in{\N}\}\subseteq{\R}^2.
$$
Set 
$$
M_a = \{(n,1/n)|n\in \N\}, \; \;   M_b = \{(n,-1/n)| n\in \N\} 
$$ 
and consider the map $\pi :X_0\to D$ defined by  $\pi (a_n)=0$
and $\pi (b_n)=1$ for each $n$.  Clearly $\pi$ is continuous but not uniformly continuous,  since the open disjoint binary cover $X_0=M_a\cup M_b$ is not uniform. 

\begin{lemma}{\rm (\cite{DP})}  In the above notation:  
\begin{itemize}
\item[(a)] the map $\pi$ is $C$-continuous for every additive closure operator  $C$ of ${\bf Unif}$ such that either  $c(M_a) \backslash M_a$ or $c(M_b) \backslash M_a$ is finite;
\item[(b)]  if $c(M_a) \backslash M_a$ is infinite for a closure operator $C$ of ${\bf Unif}$, then for every metric Baire space $B \in {\bf Unif}$ without 
isolated points  there exists a discontinuous map $f_B:\N _\infty \to B$  which is $C$-continuous. 
\end{itemize}
\end{lemma}


\begin{theorem}{\rm (\cite{DP})}   Let $C$ be an additive closure operator of ${\bf Unif}$. Then either $\pi:X_0 \to D$ is $C$-continuous  or for every metric Baire space
without isolated points $B$ there exists a discontinuous map $f_B:\N _\infty \to B$  which is  $C$-continuous.
\end{theorem}
  

This shows that 
for every additive closure operator $C$ of ${\bf Unif}$ one can find a $C$-continuous
map that is not uniformly cotninuous. Hence, a single closure operator of ${\bf Unif}$ cannot detect  uniform continuity. This theorem also allows us to see  which $C$-continuous maps fail to be u.c.

It is natural to expect that using more than just one closure operator things may change. We shall see now that this is not the case. Even a totally continuous map, which satisfies (iii) for every closure operator $C$, is not necessarily uniformly continuous.
The tool to achieve this result is the following notion introduced in \cite{DP} and \cite{BD1}:

\begin{defn}\label{UA:Definition}{\rm (\cite{BD1,DP})} We say that $f \in C(X,Y)$ is ${\bf UA}$ ({\bf uniformly approachable}), if for every compact set $K \subseteq X$ and every
set $M\subseteq X$, there is a $UC$ function $g \in C(X,Y)$ which coincides with $f$ on $K$ and satisfies $g(M) \subseteq f(M)$. 

We then say that $g$ is a {\bf $(K,M)$-approximation of $f$}.  If  we require in the definition of $UA$  that $K$ consists of a single point we obtain the weaker notion ${\bf WUA}$ ({\bf weakly UA}). \end{defn}

Clearly,  $UA$ implies $WUA$. 

It was shown in \cite{DP} that $\R$ with the natural uniformity has the property that every continuous selfmap is uniformly approachable: 
 
\begin{example}\label{ex1}
Every $f\in C(\R)$ is UA. Indeed, let $K=[-n,n]$ and let $M\subseteq \R$ be an arbitrary non-empty set. Pick any $m_1\in M\cap (-\infty, -n]$ if this set is non-empty, otherwise
take $m_1=-n$. Choose $m_2\in M$ analogously. Then the function $g:\R\to \R$ defined by

$$
g(x)=
\begin{cases}f(m_1),\mbox{ if }x\leq m_1 \\
          f(x), \;\; \mbox{ if }m_1\leq x\leq m_2\\
                f(m_2), \mbox{ if }x\geq m_2 \end{cases} 
                $$
is a $(K,M)$-approximation of $f$. 
\end{example} 

Since ``uniformly approachable'' implies ``totally continuous'' and $f(x)=x\sp 2$ is not uniformly continuous it follows that uniform continuity is not detected even by all closure operators in ${\bf Unif}$.
 However,
 Burke noticed \cite[Example~3.3]{BD1} that there are continuous non-$WUA$ functions on $\R^2$ (in fact, $f\colon\R^2\to\R$, $f(x,y)=xy$, is such a function, see Example 5).

The next theorem easily follows from the definitions:

\begin{theorem} {\rm (\cite{DP})} Every WUA function is totally continuous.
\end{theorem}



\subsection{Local view on closure operators}

Every additive and idempotent closure operator of $\bf Top$ or  $\bf Unif$ defines a topology on the underlying set of the space. 
In this sense, the use of topologies that make certain maps (uniformly) continuous in the sequel can be also viewed as a local use
of idempotent additive closure operators (i.e., on a single space or an a single pair of spaces, without the axiom (iii)). 

\subsubsection{Topologies $\tau$  on $\R$ that make a given class  of functions ${\mathcal F}\subseteq \R^\R$ coincide with $C((\R, \tau),(\R, \tau))$. }

In the sequel $X$ will be a metric space. 
 Following \cite{C} we say that a class ${\mathcal  F}$ of functions from $X$ to $Y$ {\em can be topologized} if there exist topologies $\tau_1$ on $X$ and $\tau_2$ on $Y$ such that $\mathcal  F$ coincides with
the class of all continuous functions from $(X, \tau_1)$ to $(Y, \tau_2)$. The paper \cite{C} gives conditions which under GCH (generalized continuum hypothesis) imply that
  ${\mathcal F}$ can be topologized. In particular, it is shown that (assuming GCH) there exists a connected Hausdorff topology $\tau$ on the real line such that the class of all
continuous functions in $\tau$ coincides with the class of all linear  functions. A similar theorem is valid for the class of all polynomials, all analytic functions, and all
harmonic functions. On the other hand, the classes of derivatives, $C^\infty$, differentiable, or Darboux functions cannot be topologized.

\subsubsection{Characterization of uniform continuity as a simple continuity w.r.t. appropriate topologies}

 We recall here the work of Burke \cite{B1} on characterization of uniform continuity as a simple continuity w.r.t. appropriate topologies (or {\bf locally} defined closure operators in the above sense). Within this setting, the problem becomes the question of  determining which metric spaces $X$ and $Y$ are such that the uniformly continuous maps $f\colon X\to Y$ are precisely the continuous maps between $(X,\tau_1)$ and $(Y,\tau_2)$ for some new topologies $\tau_1$ and $\tau_2$ on $X$ and $Y$, respectively. 

\begin{theorem}
There exist a connected closed subset $X$ of the plane, a homeomorphism $h: X\to X$, and a connected Polish topology $\tau$ on $X$ such that the continuous self-maps of $X$ are precisely the maps $h^n$ $(n\in\Z)$ and the constant maps, while the continuous self-maps of $(X,\tau)$ as well as the uniformly continuous self-maps of $X$ are precisely the maps $h^n$ $(n\leq 0)$ and the constant maps.
\end{theorem}

\section{Functions with distant fibers and uniform continuity}

\begin{defn} We say that $f \in C(X)$ has {\bf distant fibers} (briefly, {\bf DF}) if any two distinct fibers $f^{-1}(x)$, $f^{-1}(y)$ of $f$ are at some positive distance.
\end{defn}

It is curious to note that this property generalizes two {\em antipodal}  properties of a function:
\begin{itemize}
  \item $f$ is constant ($f$ has {\em one} big fiber)                                        
  \item  $f$ has small fibers (e.g., one-to-one functions, or more generally, functions with compact fibers, or briefly, KF). 
\end{itemize}       

\medskip

\subsection{ Uniform continuity coincides with DF for bounded functions $f:\R^m\to \R$}

It is easy to see that $UC$ implies $DF$ for any function $f:X\to\R$. Indeed, if $d(f^{-1}(u),  f^{-1}(v))=0$ for some 
$u\ne v$ in $\R$, then any pair of sequences $x_n, y_n$ such that  
$f(x_n) = u$, $f(y_n) =v$ and $\lim_n d(x_n, y_n)=0$
witness non-uniform continuity of $f$. 

Let us verify that {\em $DF$ implies $UC$ for bounded functions $f:\R^m\to [0,1]$}. Indeed, assume that the sequences $x_n, y_n$ imply non uniform continuity of $f$
with $d(x_n,y_n)\to 0$ and $|f(x_n)-f(y_n)|\geq \varepsilon$ for every $n$. Then boundedness of $f$ yields that $f(x_n), f(y_n)$ can without loss of generality  be assumed convergent, 
i.e., $(x_n)\to a, f(y_n)\to b$ for some $a\ne b$ in $[0,1]$. 
Let $I_n$ be the segment in $\R^m$ joining $x_n$ and $y_n$. 
Let $a<b$ and take $u,v\in [0,1]$ with $a<v<u<b$. Then $f(x_n)\in [0,v)$ and $ f(y_n)\in (u,1] $
for sufficiently large $n$ since $f(I_n)$ is an interval (being a connected set) containing $f(x_n)$ and $f(y_n)$. Then $u, v \in f(I_n)$ for sufficiently large $n$. Since $d(x_n,y_n)\to 0$, we conclude that $d(f^{-1}(u), f^{-1}(v))=0$, a contradiction. 

In the argument above $\R^m$ can be replaced by any space that is uniformly locally connected. This condition cannot be omitted, since the argument function on the circle minus a point is DF but not UC (indeed,  the circle minus a point is not uniformly locally connected with respect to the metric induced by the plane). 

\begin{theorem}\label{bounded}  {\rm (\cite[Theorem 3.7]{BDP1})} 
A bounded function $f \in C(X)$ on a uniformly locally connected space $X$ 
is u.c. if and only if it is $DF$. \end{theorem}

Boundedness was essential to prove that $DF$ implies uniform continuity. Indeed,  unbounded  continuous functions $\R \to \R$ need not be u.c.
even when they are {\em finite-to-one} (e.g., $x\mapsto x^2$) or even {\em injective} (e.g.,  $x\mapsto x^3$). 

The next theorem says that uniform continuity of a bounded function $f \in C(X)$ is a property of its fibers. In this form the theorem permits one to remove the 
hypothesis ``uniformly locally connected". 

\begin{theorem} \label{samefibers} {\rm (\cite[Theorem 3.10]{BDP1})} Let $(X,d)$ be a   connected and locally connected
metric space. Suppose that $f,g \in C(X, [0,1])$ have the same family of fibers and that $f$ is u.c. Then $g$ is also  u.c. \end{theorem}

Nevertheless, boundedness cannot be removed, as the pair of functions $x\mapsto x, x\mapsto x^3$ on $\R$ show. It is not clear what is the precise property of the fibres 
of $f\in C(X, [0,1])$ which gives uniform continuity.
Theorem \ref{bounded} shows that it is precisely DF when the space $X$ is uniformly locally connected. 

\bigskip

{\bf How to remove boundedness}

\bigskip

In order to remove  boundedness we now consider  a generalization of UC which  coincides with UC for bounded functions.

We start with a notion which is stronger than DF. A function $f$ is said to be {\bf proper} (briefly, P) if the $f$-preimage of any compact set is compact. 
Equivalently, $f$ is a closed map with KF. Even though in general the implication $P\to KF$ cannot be inverted (e.g., $x\mapsto \arctan x$ in $\R$), one can prove that 
{\em $P=KF$ for unbounded functions $\R^m\to \R$, $m>1$}. In particular, this holds for polynomial functions $\R^m\to \R$. 

The property $KF$ implies $DF$, but it is much stronger. Indeed, $UC$ need not imply $KF$. This is why we introduce 
the auxiliary notion $AP$ that presents a weakening of both the notion of proper function and that of UC function.
We plan to show that $DF = AP$ for functions $f \in C(X)$ on a uniformly locally connected space $X$. 

\begin{defn}{\rm (\cite{BDP1})} $f \in C(X,Y)$ is said to be $AP$ ({\bf almost proper}) if $f$ is u.c. on the $f$-preimage of every compact set. \end{defn}

Obviously, UC implies AP, whereas bounded AP functions are UC. The same argument given above to prove $UC \to DF$ also proves  that $AP$ implies $DF$. On the other hand, 
with the proof of Theorem \ref{bounded}  outlined above one can also show: 

\begin{lemma}\label{DF>AP}{\rm (\cite[Lemma 3.5]{BDP1})} $DF \to AP$ for functions $f \in C(X)$ on a uniformly locally connected space $X$. Hence, 
 $AP$ coincides with $DF$ on uniformly locally connected spaces. \end{lemma}
 
 In this way we have achieved our goal by replacing  UC with  AP. 
 
 Next we discuss an alternative solution, based on a different idea of choosing instead of AP  a class of functions close to UC in the sense of approximation, namely $UA$ (and $WUA$).

\begin{theorem}{\rm (\cite[Theorem 3.15]{BDP1})}\label{Thh}
DF implies  UA in uniformly locally connected spaces.
\end{theorem}

The proof is based on the notion of {\bf truncation}, which was implicit in Example \ref{ex1}. Now we define a different kind of  truncation (for the general definition see Definition \ref{def_truncation}).  For a function $f:X\to \R$ and real numbers $a\leq b$ define the $(a,b)$-{\bf truncation} as follows:

 $$f_{(a,b)}(x)=\begin{cases}f(x),\mbox{ if }f(x)\in [a,b] \\
 a, \mbox{ if } f( x)\leq a     \\
                b \mbox{ if }f(x)\geq b 
                \end{cases}.$$

One proves that $f\in DF$ implies $f_{(a,b)}\in DF$, hence $f_{(a,b)}\in UC$ since it is  bounded. Now, if $K$ is a compact set and $a, b\in \R$ are chosen such that $f(K)\subseteq [a,b]$, then $f_{(a,b)}$ is a $(K, M)$-approximation of $f$ if one takes additional care about $g(M) \subseteq f(M)$ as in Example 1. 

The above implication cannot be inverted, as the next example shows. 

\begin{example}\label{sin_x}
The function $f(x)=\sin x^2$ is not DF, since any two fibers of $f$ are at distance 0. Nevertheless, according to Example  \ref{ex1}, $f(x)$ is UA. 
\end{example}

 This suggests that the condition DF is too strong. We shall consider
an appropriate weaker version below. 

\begin{rem} Note that the $(a,b)$-truncation is different from the truncation $g$ defined in Example
\ref{ex1} in the case $X=\R$. It can  easily be shown that if $f([m_1,m_2])=[a,b]$, then  $g$ is a truncation of $f_{(a,b)}$. 
\end{rem}

\subsection{Distant connected components of fibers}


In order to invert the implication in Theorem \ref{Thh} we need a weaker form of DF. To this end we take a closer look at the {\em connectedness} structure of the fibers. 
First of all we recall a notion for continuous maps in $\bf Top$. 

\begin{defn} A continuous maps $f: X \to Y$ between topological spaces is called 
\begin{itemize}
\item[(a)] {\em monotone} if all fibers of $f$ are connected;
\item[(b)]  {\em light} if all fibers of $f$ are totally disconnected. 
\end{itemize}
\end{defn}

The term in item (a) was motivated by the fact that for maps $\R \to \R$ one obtains the usual monotone maps. 
It is known that every continuous map $f: X \to Y$ between topological spaces can be factorized as $f= l\circ m$, where $m: X\to Z$ is monotone and $l: Z \to Y$ is light. 
(Notice that this factorization for the constant map $f: X \to Y=\{y\}$ provides as $Z$ exactly the space of connected  components of $X$ and $m: X \to Z$ is the quotient
map having as fibers the connected components of $X$.)

\begin{defn}
We say that $f$ has {\bf ``distant connected components of fibers''} ({\bf DCF}) in the sense that any two components of distinct fibers are at positive distance. 
\end{defn}

\begin{example} Let $X$ be a metric space and  $f: X \to \R$ a continuous map. 
\begin{itemize}
\item[(a)] if $f$ is monotone, then $f$ is DCF if and only if $f$ is DF. 
\item[(b)] If $f$ is light, then it is DCF.   
\end{itemize}
In particular, the function $f: (0,1/\pi] \to \R$ defined by $f(x) = \sin 1/x$, as well as the function from Example \ref{sin_x}, are DCF (being light), but any two non-empty fibers of $f$ are at distance 0. 
\end{example}

\begin{theorem}{\rm (\cite[Theorem 4.3]{BDP1})} UA implies DCF for $f\in C(X)$ and arbitrary metric spaces $X$. 
\end{theorem}

In fact, if $C_a$ and $C_b$ are two connected components of fibres of $f$ at distance 0, then for $K=\{a,b\}$ and $M=C_a\cup C_b$  the function $f$ has no $(K,M)$-approximation. 

Along with Example  \ref{ex1} this gives: 

\begin{cor}\label{new}
Every continuous real-valued function $\R \to \R$ is DCF. 
\end{cor}

Actually, we shall see below that even $WUA$ implies $DCF$ for $f\in C(\R^m)$ (see Theorem \ref{*}), therefore,

$$
DF\to UA\to WUA \to DCF \mbox{ for }f\in C(\R^m).
$$


Hence all they coincide for polynomial functions (or functions with finitely many  connected components of fibres). Let us put all these implications in the following diagram, 
where the equivalence (1) for ULS spaces is given by Lemma \ref{DF>AP} and the implication (2) for ULS spaces is given by Theorem 6. The implication (3) for $\R^n$ follows from these two implications
and the trivial implication $UA 	\to WUA$. The implication (4) will be proved in Theorem \ref{*} below which gives a much stronger result. 

\vskip-190pt

\begin{center}
\begin{picture}(362,100)
\put(80,30){\vector(0,1){35}}
\put(80,58){\vector(0,-1){35}}
\put(80,-15){\vector(0,1){25}}
\put(98,72){\vector(1,0){40}}
\put(95,17){\vector(1,0){45}}
\put(40,17){\vector(1,0){30}}
\put(15,5){\makebox(30,15)[t]{KF}}

\put(174,72){\vector(1,0){38}}

\put(154,60){\vector(0,-1){35}}
\put(154,25){\vector(0,1){35}}
\put(140,5){\makebox(30,15)[t]{UA}}
\put(105,-2){\makebox(30,15)[t]{(2)}}
\put(105,55){\makebox(30,15)[t]{(3)}}
\put(175,55){\makebox(30,15)[t]{(4)}}  
\put(140,60){\makebox(30,15)[t]{WUA}}
\put(65,60){\makebox(30,15)[t]{AP}}
\put(65,5){\makebox(30,15)[t]{DF}}
\put(65,-35){\makebox(30,15)[t]{UC}}
\put(75,35){\makebox(30,15)[t]{(1)}}
\put(210,60){\makebox(30,15)[t]{DCF}}
\end{picture}
\end{center}

\vskip47pt

\begin{center}
Diagram 1
\end{center}

\vskip15pt

We shall see below that (4) is not an equivalence.  This motivated the introduction of the following weaker version of $UA$ in \cite{CD2}: a function $f\colon X\to\reals$ is 
said to be $UA_d$ ({\em densely uniformly approachable}) if it admits uniform  $\la K,M\ra$-approximations for every {\em dense\/} set $M$ and for every compact set $K$. 
Analogously, one can define $WUA_d$.

\begin{theorem}\label{*} 
$WUA_d$ coincides with $DCF$ for $f\in C(\R^m)$. 
\end{theorem}

The proof requires a new form of weak UC based on truncations:

\begin{defn}\label{def_truncation} $g \in C(X)$ is a {\bf truncation} of $f \in C(X)$ if the space $X$ can be partitioned in two parts $X = A \cup B$ so 
that $g = f$ on $A$ and $g$ is constant on each connected component of $B$ (that is, $g$ must be constant on each connected component of 
$\{x\in X: f(x) \neq g(x)\}$). 
\end{defn}

This motivates the introduction of the class $TUA$ of {\em truncation-$UA$\/} functions, that is, functions $f \in C(X)$ such that for every compact set $K \subseteq X$ there is a u.c.
truncation $g$ of $f$ which coincides with $f$ on $K$.

The following is easy to prove:

\begin{theorem}\label{TUA>DCF}
{\rm (\cite{BDP1})} TUA implies DCF on every locally connected space. 
\end{theorem}

Indeed, if $C_a$ and $C_b$ are two connected components of fibres of $f$ at distance 0, then for $K=\{a,b\}$ the function $f$ has no UC $K$-truncations. 

The proof of Theorem \ref{*} splits in three steps (see in Diagram 2 below)

\begin{itemize}
   \item {\bf Step 1:} (\cite[Corollary 7.4]{BDP1}) $DCF\to TUA$  for $f\in C(\R^m)$
   \item {\bf Step 2:} (\cite[Theorem 3.1]{CD2}) $TUA \to UA_d$  for $f\in C(\R^m)$
   \item {\bf Step 3:} (\cite{BDP1}, \cite[Corollary 4.2]{CD2}) $WUA_d\to DCF$ for $f\in C(\R^m)$. 
\end{itemize}

Step 1 and Theorem \ref{TUA>DCF} ensure the equivalence (4) for  $f\in C(\R^m)$ in Diagram 2. 
Step 2, the trivial implication $UA_d \to WUA_d$, Step 3 and  the equivalence (4) imply (5). 
This proves all four equivalences for  $f\in C(\R^m)$ in the right square of Diagram 2. 

The remaining three implications (1), (2) and (3) are  trivial.  

\vskip-65pt

\begin{center}
\begin{picture}(362,100)
\put(80,30){\vector(0,1){35}}
\put(98,72){\vector(1,0){40}}
\put(95,17){\vector(1,0){45}}

\put(174,72){\vector(1,0){38}}
\put(172,17){\vector(1,0){38}}

\put(154,60){\vector(0,-1){35}}
\put(154,25){\vector(0,1){35}}
\put(210,72){\vector(-1,0){37}}
\put(210,17){\vector(-1,0){40}}
\put(220,30){\vector(0,1){30}}
\put(220,60){\vector(0,-1){30}}
\put(140,5){\makebox(30,15)[t]{UA$_d$}}
\put(105,-2){\makebox(30,15)[t]{(2)}}
\put(105,55){\makebox(30,15)[t]{(3)}} 
\put(175,-2){\makebox(30,15)[t]{(5)}} 
\put(140,60){\makebox(30,15)[t]{WUA$_d$}}
\put(65,60){\makebox(30,15)[t]{WUA}}
\put(65,5){\makebox(30,15)[t]{UA}}
\put(75,35){\makebox(30,15)[t]{(1)}}
\put(215,35){\makebox(30,15)[t]{(4)}}
\put(210,60){\makebox(30,15)[t]{DCF}}
\put(212,5){\makebox(20,15)[t]{TUA}}
\end{picture}
\end{center}

\vskip7pt

\begin{center}
Diagram 2
\end{center}

\vskip15pt


In view of the four equivalences in the right square of  Diagram 2, the next example shows that the implications (2) and (3) cannot be inverted. 

\begin{example}{\rm (\cite[\S 5]{CD2})}
In $C(\R^2)$, TUA does not imply WUA.
\end{example}

It remains unclear whether the remaining last implication (1) of  Diagram 2 can be  inverted for $f\in C(\R^m)$ (see Problem \ref{UA>WUA}). 

\section{$UA$ spaces} 

 The main objective of this section are the {\bf  UA spaces} -- spaces where every continuous function is UA. 
  The first example of this kind is $\R$ (Example \ref{ex1}). The motivation to introduce these spaces are the well known  Atsuji spaces.

Here we recall some results from \cite{BD1} and we anticipate some of the pricipal results from \cite{BDP2} which give further motivation for studying UA functions.

The next definition will be used in the sequel. 

\begin{defn} \label{separated} Two subsets $A, B$ of a topological space $X$ are said to be {\bf separated} if the closure of each of them does not meet  the other (this is equivalent to saying that $A$ and $B$ are clopen in $A \cup B$). So $X$ is connected if and only if  it cannot be partitioned in two separated sets. 
 
A subset $S$ of $X$ {\bf separates} the nonempty sets $A$ and $B$ if the complement of $S$ can be partitioned in two separated sets, one of which contains $A$, the other contains $B$ (see \cite[\S 16, VI]{K}). 
\end{defn}

\subsection{UA spaces}

Several criteria for UA-ness are given, among them the following looks most spectacular: 

\begin{theorem}\label{ostacolo} {\rm (\cite{BDP2})} Let $X$ be a $UA$ space and let $A, B$ be disjoint closed uniformly connected subsets of $X$. Then there is a collection $\{H_n \mid n \in \N\}$ of nonempty closed subsets of $X$ such that for every $n$, 
\begin{enumerate}
\item $H_{n+1} \subseteq H_n$;
\item $H_n$ separates $A$ and $B$; and  
\item $H_n$ is contained in a finite union of balls of diameter $< 1/n$. 
\end{enumerate}
\end{theorem}
 
 Actually, this property can be proved for a larger class of spaces discussed in \S 6, where a relevant property
 is obtained in the case when $X$ is complete (Theorem \ref{main}).


The following construction which produces UA spaces from trees and compact sets placed at their vertices was given in \cite{BDP1}. 

\begin{defn}\label{tree}  A metric space $X$ is a {\em tree of compact sets $\{K_n : n \in \omega\}$} if $X = \bigcup_{n \in \omega}K_n$ where each $K_n$ is compact and $|K_{n+1} \cap \bigcup_{i \leq n}K_i|=1$.
\begin{itemize}  
\item[(a)]
 Given 
 a subset $I\subseteq \omega$, we say that  the subspace $X_I= \bigcup_{n \in I} K_n$ of $X = \bigcup_{i \in \omega} K_i$  is a {\em subtree} of $X$ if for every $n, m$ with $n<m$, if $n \in I$ and $K_n \cap K_m \neq \emptyset$, then $m \in I$.   
\item[(b)] 
 A tree of compact sets $X = \bigcup_{i \in \omega} K_i$ is said to be {\em tame} if every $K_i$ has an open neighbourhood which intersects only finitely many $K_j$'s  and every two disjoint subtrees of $X$ are at a distance $>0$.
 \end{itemize}
\end{defn}

It is easy to see that the circle minus a point can be represented as a tree of compact sets, but none of these  trees is tame. The next theorem shows the reason for that
(the circle minus a point is not a UA space).  

\begin{theorem} {\cite{BDP2}}
 If $X = \bigcup_{i\in \omega} K_i$  is a tame tree of compact sets $\{K_n : n \in \omega\}$, then $X$ is $UA$.
\end{theorem}

Examples of tame trees are given in Figure 1 (see ladders B and C). 

\subsection{Non-UA spaces: Hedghogs and some necessary conditions}

Let $\alpha$ be a cardinal. In the sequel $H_\alpha$ denotes the {\em hedghog} with $\alpha$ spikes (see \cite[Example 4.1.5]{E},  note that $H_a$ is separable if and only if $a=\omega$). 

 Recall the definition of the cardinal $\mathfrak b$ as the  minimal cardinality of an unbounded family of functions $f: \omega\to\omega$ with respect to  the partial {\it pre}order  $f\leq^* g$  if $ f(n)\leq g(n)$  for all but finite number  $n\in \omega$ (see \cite{vD}). In ZFC $\omega_1\leq \mathfrak b \leq 2^\omega$,  and $\mathfrak b = 2^\omega$ consistently (for  example under  MA or CH, see \cite{vD} for more detail). 
 
 Surprisingly, one has the following independency result: ZFC cannot decide whether the smallest non-separable hedghog $H_{\omega_1}$ is  UA. More precisely, the following holds: 
 
\begin{theorem}\label{Hedg} {\rm \cite{BDP2}}
Let $\alpha$ be a cardinal. If $\alpha<\mathfrak b$ then $H_\alpha$ is UA, whereas if $\alpha\geq \mathfrak b$ then $H_\alpha$ is not even $WUA$.
\end{theorem} 
 
 In particular, under CH the space $H_{\omega_1}$ is not UA, while in models of ZFC where 
 $\neg CH \& MA$ holds, one has ${\omega_1}<\mathfrak b$, so $H_{\omega_1}$ is $UA$. 

We show below that every space $H_{\alpha}$ is $TUA$ (Corollary \ref{corSTUA}). Hence $TUA\not \to WUA$ for nonseparable spaces. 

We  also  consider the following space which is more general than  the hedghog $H_\alpha$ of $\alpha$ spikes:

\begin{defn}\label{star} 
Let $\alpha$ be an infinite cardinal. A metric space $(X, d)$  is called a {\em hedghog of compact sets
 $\{K_\lambda :  \lambda \in \alpha\}$} if $X =\bigcup_{\lambda \in \alpha}K_\lambda$
where each $K_\lambda$  is compact with more than one point and there exists $p\in X$ such that 
\begin{enumerate}
     \item  $K_{\lambda} \cap K_{\lambda'}=\{p\}$ for $\lambda \neq \lambda'$; and 
     \item $  \forall x,y \in X [ d(x,y)<\max\{ d(x,p), d(y,p)\}\to   \exists \lambda<\alpha [x, y \in K_\lambda]]$.
\end{enumerate}
 
 Sometimes we prefer to say more precisely: {\em a hedghog of $\alpha$ compact sets}.

 \end{defn}

\begin{defn}\label{min_truncation} Let $X$ be a uniform space, let $f\in C(X)$ and let $K$ be a compact subset  of the space $X$. The {\it minimal $K$-truncation} $f_K$ of $f$ is defined  as the $(\inf _Kf, sup _Kf)$-truncation of $f$.
\end{defn}


\begin{theorem}\label{STUA} {\rm  \cite{BDP2}} Let $X$ be a hedghog of compact sets and  $K$ be a compact subset of $X$ containing $p$.  Then for every continuous function  $f\in C(X)$ 
the $(\inf _Kf, sup _Kf)$-truncation of  $f$  is u.c. 
\end{theorem}

This gives the following: 

\begin{cor}\label{corSTUA}  Every hedghog of compact sets  is $TUA$. \end{cor}

Another source of UA spaces is given by the following: 

 \begin{theorem} {\rm  \cite{BDP2}} Every uniformly zero-dimensional space is UA. \end{theorem} 
 
 In [BDP3] the Cantor set is  characterized as the only compact metrizable space $M$ such that each subspace of $M$ is UA. 

 \begin{theorem} {\rm  \cite{BDP2}}  The only  manifolds which are WUA are the compact ones and the real line. \end{theorem}

Now we see that a metric space having a continuous  function which is not  uniformly continuous, necessarily also  has a 
bounded uniformly approachable  function that is not uniformly continuous. Hence, in some sense, 
UA is ``closer to continuity than to uniform continuity". 

\begin{theorem}\label{UC} {\rm  \cite{BDP2}} A metric space  $X$ is $UC$ if and only if every bounded uniformly approachable  function is uniformly continuous, so that a space $X$ with $C_u(X) =C_{ua}(X)$ is necessarily  a $UC$ space. 
\end{theorem}


In order the get a necessary condition for being a WUA space, the following notion was proposed in [BeDi1]: 

\begin{defn} Let $X$ be a uniform space. A family of {\em pseudo-hyperbolas} in $X$ is given by 
a countable family $\{H_n\}$ of disjoint subsets of $X$ such that for every $n \in \N$:  
\begin{enumerate}
\item $H_n$ is closed and uniformly connected; 
\item $H_n \cup H_{n+1}$ is uniformly connected;
\item $H_n \cap \overline{\bigcup_{m>n} H_n} = \emptyset$; and
\item 
the set $H = \bigcup_n H_n$ is not closed in $X$.
\end{enumerate}
\end{defn}

This notion was inspired by the following example due to Burke. 

\begin{example} A family of pseudo-hyperbolas in ${\bf R}^2$ is given  by the sets $H_n = \{(x,y) : (xy)^{-1} =n\}$. \end{example}

\begin{theorem}\label{ps-hyp} {\rm (\cite{BD1})} If a normal uniform space $X$ has a family of pseudo-hyperbolas, then $X$ is not $WUA$. \end{theorem}


\begin{theorem}\label{nonWUAmagic} {\rm (\cite{BD1})} Let $X$ be a separable uniform space and suppose that there exists $f \in C(X)$ with
countable fibers without non-constant uniformly continuous truncations. Then $X$ is not $WUA$. \end{theorem}

\begin{example}
Now we give three non-$WUA$ examples of subsets $X_1, X_2$ and $X_3$ of $ \R^2$ that contain no pseudo-hyperbolas. 
 To prove
that they are not $WUA$ one can apply Theorem \ref{nonWUAmagic}. So it is necessary to find, in each case, a continuous
function with countable fibers and without non-constant uniformly continuous truncations. 
\begin{itemize}
\item[(a)] The space $X_1$ is the unit circle minus a non-empty finite set. So 
 $X_1$  can be   identified with a cofinite subset of the set of complex numbers $e^{i\theta}$ with $0<\theta < 2\pi$. Define $f \colon X_1
\to \R$ by $f(e^{i\theta}) = \theta$. Then $f$ is non-$WUA$. It is easy to see that $X_1$ contains no pseudo-hyperbolas. 
\item[(b)] The space $X_2$ consists of the union of the two hyperbolas 
$H_1= \{(x,y\in \R^2: x\geq 0, y \geq 0, xy=1\}$ and $H_2= \{(x,y\in \R^2: x\geq 0, y \geq 0, xy=2\}$. Obviously, $X_2$ contains no pseudo-hyperbolas. 
Define $f \colon X_2 \to \R$ as follows. If $(x,y)\in H_1$, then set $f(x,y) = e^x$. If $(x,y)\in H_2$, then set $f(x,y) = -e^{-x}$. It is easy to see that this works.
\item[(c)] Let $X_3$ be the space
from Diagram 3. It contains no pseudo-hyperbolas. 
 Define $f \colon X_3 \to \R$ by identifying $X_3$ with the subspace of $\R^2$ consisting of the union of the two vertical axes
$x = -1$ and $x = 1$, together with the horizontal segments $I_n =\{(x, n) \in \R^2 : -1 \leq x \leq 1\}$ ($n \in \N$).  Let $f(-1, y) =
-y$, $f(1, y) = {y}$. This defines $f$ on the two axes of the ladder. On each horizontal segment $I_n$, $f$ is linear.
This  uniquely defines $f$ since we have already defined $f$ on the extrema of the horizontal segments $I_n$. $f$ is pseudo-monotone, so each truncation of $f$ is an
$(a,b)$-truncation.  Any such non-constant truncation is not uniformly continuous. Since $f$ has countable
fibers, Theorem \ref{nonWUAmagic} applies and thus $X_3$ is not $WUA$. 
\end{itemize}
\end{example}

\begin{center}
\begin{picture}(341,211) \label{ladder}
\put(140,195){\line(0,-1){130}} 
\put(160,195){\line(0,-1){130}} 
\put(140,65){\line(1,0){20}} 
\put(140,85){\line(1,0){20}} 
\put(140,105){\line(1,0){20}} 
\put(140,125){\line(1,0){20}} 
\put(140,145){\line(1,0){20}} 
\put(140,165){\line(1,0){20}} 
\put(140,185){\line(1,0){20}} 

\end{picture} 
\end{center}

\vskip-0.8in

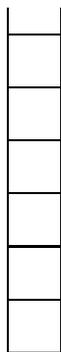
\centerline{Diagram 3: A subset of $\R^2$}

\bigskip


%



\section{Thin spaces}

The class of topological spaces $X$ having connected quasi components  is closed under homotopy type and it contains all compact Hausdorff spaces 
(see \cite[Theorem 6.1.23]{E}) and every subset of the real line. Some sufficient conditions are given in \cite{GN} (in terms of existence of 
Vietoris continuous selections) and \cite{CMP} (in terms of the quotient space $\Delta X$ in which each quasi-component is identified to a point), 
but an easily-stated description of this class does not seem to be available (see \cite{CMP}). The situation is complicated even in the case when all 
connected components of $X$ are trivial, i.e., when $X$ is hereditarily disconnected.   In these terms the question is  to distinguish between hereditarily disconnected  and totally disconnected spaces (examples to this effect go back to Knaster and  Kuratowski \cite{KK}).  

The connectedness of the quasi component (i.e., the coincidence of the quasi component and the connected component) in topological groups is also a
rather hard question.  Although a locally compact space does not need to
have connected quasi components \cite[Example 6.1.24]{E}, all locally compact groups 
have this property. This is an easy consequence of the well known fact that the connected component of a locally compact group coincides with
the intersection of all open subgroups of the group \cite[Theorem 7.8]{HR}.  All countably compact groups were shown to have this property, too (\cite{D3}, see also \cite{D2,D4}). Many examples of pseudocompact group where this property strongly fails in different aspects, as well as further information on quasi components in topological groups, can be found in (\cite{D1,D2,D4}, see also \cite{U} for a planar group with non-connected quasi components). 

Let us recall the definition of the {\em quasi component} $Q_x(X)$ of a point $x$ in a topological space $X$. This is the set of all points $y\in X$ such that $f(y) = f(x)$ for every continuous function
$f: X\to \{0,1\}$, where the doubleton $\{0,1\}$ is discrete. 
Analogously, given a uniform space $X$ and a point $x \in X$ the {\em uniform quasi component} of $x$
consists of all points $y\in X$ such that $f(y) = f(x)$ for every uniformly continuous function
$f: X\to \{0,1\}$, where the doubleton $D=\{0,1\}$ has the uniformly discrete structure (i.e., the 
diagonal of $D \times D$ is an entourage). We denote by $Q^u_x(X)$ the uniform quasi component of $x$. 

In these terms we have the following inclusions 
$$
C_x(X) \subseteq Q_x(X) \subseteq Q^u_x(X), \eqno(*)
$$
where $C_x(X)$ denotes the connected component of $x$. Now we can introduce the relevant notion for this section: 

\begin{defn}\label{Separates} A uniform space $X$ is said to be  {\bf thin} if for every closed subset $Y$ of $X$ and every $y \in Y$, the uniform quasi component of $y$ in $Y$ is connected.  \end{defn}

It is easy to see that all inclusions in (*) become equalities in the case of
compact spaces. Hence compact spaces are thin. This  also follows
from the more general property given in Theorem \ref{UA>Thin}.

\begin{defn} For three subsets $A, B$ and $S$ of a topological space $X$ we say that $S$ {\bf cuts} between $A$ and $B$ if 
$S$ intersects every connected set which meets both $A$ and $B$. (If $S$ is empty this means 
that there is no connected set which meets both $A$ and $B$.)
\end{defn}

If a set separates $A$ and $B$ (see Definition \ref{Separates}), then it also cuts between $A$ and $B$, but the converse is false in general. 

\begin{defn} We say that a uniform space $X$ has the {\bf compact  separation property} (briefly $\CSP$), if for any two disjoint closed 
connected subspaces $A$ and $B$ there is a compact set $K$ disjoint from  $A$ and $B$ such that every neighbourhood of $K$ disjoint from $A$ and $B$
separates $A$ and $B$ (consequently $K$ intersects every closed connected set which meets both $A$ and $B$, see Definition \ref{separated}).
\end{defn}

It is easy to see that every compact space has CSP since disjoint compact sets are always separated. 

It was proved in  \cite[Lemma 3.2]{BDP2} that if a metric space $X$ contains two disjoint closed sets $H$ and $K$ and a point $a \in H$ such that the uniform quasi
component of $a$ in $H \cup K$ intersects $K$, then $X$ is neither thin nor $UA$. 
This yields the following corollary: 

\begin{cor}\label{positive} Two disjoint closed uniformly connected subsets $A, B$ of a thin metric space $X$ are at positive distance.
\end{cor}

The following notion is relevant to the description of thin spaces. 

\begin{defn} \label{garland-configuration} Given two  distinct points $a,b$ of a metric space $X$ such that the uniformly connected component of $a$ contains $b$, there exists  for each $n$ a finite set $L_n \subset X$ whose points form a $1/n$-chain from $a$ to $b$. We say that the sets $L_n$, together with $a$ and $b$, form a ({\bf discrete}) {\bf garland}, if there 
is an open subset $V$ of $X$ which separates $a$ and $b$ and such  that $V\cap\bigcup_n L_n$ is closed (and discrete). \end{defn}

One can give a characterization of the thin of metric spaces in terms of existence of garlands in the space. 

\begin{prop}\label{new_prop} {\rm \cite{BDP2} } For a metric space $X$ the following conditions are equivalent: 
  \begin{itemize}
     \item[(a)] $X$ is thin;
     \item[(b)] $X$ contains no garlands;  
     \item[(c)] $X$ contains no discrete garlands.
\end{itemize}
\end{prop}

The main result of the  paper \cite{BDP2} is the following:

\begin{theorem}\label{thmThin1}
 Every complete thin metric space has $\CSP$. 
 \end{theorem}

A large source of thin spaces is provided by US spaces. 

\begin{theorem}\label{UA>Thin} {\rm (\cite{BDP2})} Every UA metric space space is thin. 
 \end{theorem}

 Theorem \ref{thmThin1} follows from the following more precise result:

\begin{theorem} \label{main} Let $X$ be a complete thin metric space  and let $A, B$ be disjoint closed connected subsets of $X$. Then:
\begin{enumerate}
  \item there is a compact set $K$ such that each \nbd of $K$   disjoint from $A \cup B$ separates $A$ and $B$;
  \item hence $K$ intersects every closed connected set which meets $A$ and $B$;
  \item if $X$ is also locally compact, there is a compact set $K'$ which separates $A$ and $B$. 
\end{enumerate}
\end{theorem}

\subsection{CSP vs thin and complete}

The  next examples show that the implications in Theorems \ref{thmThin1} and \ref{UA>Thin} cannot be inverted. 

\begin{example}
There exist many examples of separable metric space with  CSP which are not thin: 
  \begin{itemize}
    \item[(i)] the circle minus a point (it has two closed connected subsets at distance zero, so it cannot be thin by Corollary \ref{positive});
    \item[(ii)] the rationals $\Q$ (uniformly connected non-connected, hence not thin).
\end{itemize}
\end{example}

None of the above examples is complete. Here is an example of a complete separable metric space with CSP which is not thin. 

\begin{example}\label{last} 
Let $H_1$ and $H_2$ be the branches of hyperbolas $\{(x,y)\in \R^2: xy=1\}$  and $\{(x,y)\in \R^2: xy=2\}$, respectively, contained in the first quadrant. 
Then the space $X=H_1\cup H_2$ with the metric induced from $\R^2$ is a  complete separable space. Since  $H_1$ and $H_2$ are
 connected and at distance zero, it follows from Corollary \ref{positive}  that $X$ is not thin. On the other hand, the empty set 
separates the closed connected sets $H_1$ and $H_2$. So if $A$ and $B$ are  closed connected disjoint sets in $X$, it remains to consider only the  case when 
both $A$ and $B$ are contained in the same component $H_i$ ($i=1,2$). Now $A$  and $B$ can be separated by a point.  \end{example} 

Theorem \ref{thmThin1} can be given the following more general form. A metrizable space $X$ with compatible metrics $d_1, d_2$ such that $(X,d_1)$ is  complete (i.e. $X$ is \v{C}ech-complete) and $(X,d_2)$ is thin admits also a compatible metric $d$ such that  $(X, d)$ is complete and thin (namely, $d=max\{d_1, d_2\}$).
Hence every \v{C}ech-complete metrizable space that admits a compatible thin metric has CSP. This explains why the spaces in (i) above and Example \ref{last} have CSP. 

Although completeness was essentially used in the proof of Theorem \ref{thmThin1}, it is not clear whether it is in fact necessary, in other words:

\begin{question}\label{non-complete}
 Are there examples of thin spaces that do not have CSP? What about $UA$ spaces? 
\end{question}

As the following example shows, neither thinness nor $UA$-ness is preserved by passing to completions, thus an immediate application of Theorem \ref{thmThin1} (via passage 
to completions) cannot help attempts to answer Question \ref{non-complete}. 

\begin{example}
There is a $UA$ metric space whose completion is not thin (hence not $UA$).  Let $X = \bigcup_{n\in \N} \{1/n\} \times I$,  where $I$ is the unit interval 
$[0,1]\subset {\bf R}$, let $a=(0,0)$, $b=(0,1)$ and $Y=X\cup \{a,b\}$. We put on $Y$ the following metric. The distance between two points $(x_1,y_1)$ and $(x_2,y_2)$ is $|y_1-y_2|$ if $x_1=x_2$. Otherwise the distance is the minimum between  $y_1+y_2 + |x_1 - x_2|$ and $(1-y_1) + (1-y_2) +|x_1 - x_2|$. With this metric $Y$ is the completion of $X$ and the two points $a, b$ are the limits for $n \to\infty$ of $(1/n, 0)$ and $(1/n, 1)$, respectively. The space $Y$ is not thin since there is a garland consisting of $a$, $b$ and $\langle L_n \mid n \in {\bf N}\rangle$ where $L_n$ is a $1/n$-chain between $a$ and $b$ in $\{1/n\} \times I$. The space $X$ is $UA$ since $X$ is a union of a chain of compact sets, each 
attached to the next by at most one point (see \cite[Theorem 11.4]{BD1} and the introduction).  \end{example}

\subsection{Thin does not imply UA for complete metric spaces}

We give an example of a complete connected thin metric space that is not $UA$. 

\begin{example} For any cardinal $\alpha$ 
 the hedgehog  $J(\alpha)$  
is thin. Indeed, if $J(\alpha)$ were not thin, then by Proposition \ref{new_prop}, it would contain a discrete garland $a,b, \langle L_n \mid n \in {\bf N}
\rangle$. Let $V$ be an open set separating $a,b$ such that $V \cap \bigcup_n L_n$ is closed and discrete.  The minimal connected set $C$ containing $a,b$ must non-trivially 
intersect $V$, so it contains an open interval $I$ on one of the spikes. Now, whenever $1/n$ is less than the diameter of $I$, $L_n$ must intersect $I$, so $V \cap \bigcup_n L_n$ has an accumulation point, which is a contradiction. 
\end{example}

This gives the following immediate corollary of Theorem \ref{Hedg}

\begin{cor}
For every $\alpha \geq {\mathfrak b}$ the hedgehog $J(\alpha)$ is thin (so has the property CSP), but not UA. 
\end{cor}

The space $J(\alpha)$ is not separable for $\alpha>\omega$. On the other  hand, $\mathfrak b>\omega$ (\cite{vD}), hence the above examples are not 
separable. According to Theorem \ref{Hedg}  the hedgehogs $J(\alpha)$  are $UA$ for all $\alpha<\mathfrak b$, so one cannot get in this way an example of a separable space with the above properties  (see Question \ref{Ques_Thin}). 

\section{Gluing uniformly continuous functions}

It is well-known fact that a map $f: X \to Y$ between topological spaces is continuous whenever its restriction to each member of a locally finite closed cover of $X$
 is continuous. This section is dedicated to the analogue of this property for uniform continuity. 

\subsection{Straight spaces}

In order to characterize the spaces where uniformly continuous functions can be glued as the continuous ones, the following definition was introduced
in \cite{BDP3}: 

\begin{defn}
A space $X$ is called {\bf straight} if whenever $X$ is the union of two closed sets, then $f \in C(X)$ is u.c. if and only if its restriction to each of the closed sets is u.c.
\end{defn}

Apparently, it would be more natural to ask about the possibility to glue together {\em finite} number of u.c. functions instead of just two.
The following geometric criterion obtained in \cite{BDP3} justifies  this choice. 

Two subsets $A$ and $B$ of a uniform space $X$ are called  {\em $U$-distant} (or simply, {\em distant}) if there exists 
an entourage $U$ such that $A[U]\cap B=\emptyset$ (or equivalently,
there exists an entourage $U$ such that $A\cap B[U]=\emptyset$).

\begin{defn}\label{u_placed} Let $(X,{\mathcal U})$  be a uniform  space. A pair $C^+,C^-$ of closed sets  of $X$ is said to be  {\em u-placed} 
if $C^+_U$ and $ C^-_U$ are distant for every entourage $U$, where  
$
C^+_U=\{x\in C^+| x\not\in (C^+\cap C^-)[U]\}$
$
C^-_U=\{x\in C^-| x\not\in (C^+\cap C^-)[U]\}.
$
 \end{defn}

\begin{rem}\label{unif_clop1}
\begin{itemize}
\item[(a)] In the case of a metric space $(X,d)$ we always consider the metric uniformity of $X$, so that in such a case a
 pair $C^+,C^-$ of closed sets  of $X$ is {\em u-placed} if $d(C^+_\eps, C^-_\eps)>0$ holds
for every $\eps>0$, where $C^+_\eps=\{x\in C^+: d(x,C^+\cap C^-)\geq \eps\}$ and $C^-_\eps=\{x\in C^-: d(x,C^+\cap C^-)\geq \eps\}$.  
\item[(b)] Note that $C^+_\eps=C^+$ and $C^-_\eps=C^-$ when $C^+\cap C^-= \emptyset$ in Definition \ref{u_placed}. Hence a partition $X=C^+\cup C^-$ of $X$
into clopen sets is u-placed if and only if $C^+,C^-$ are uniformly clopen (a subset $U$ of a uniform space $X$ is {\it uniformly clopen} if the characteristic function $X\to
\{0,1\}$ of $U$ is uniformly continuous where $\{0,1\}$ is discrete).
  \end{itemize}
   \end{rem}

\begin{theorem}\label{glue}  For a uniform space $(X,{\mathcal U})$ and a pair $C^+,C^-$ of closed sets the following statements are equivalent:
\begin{enumerate}
\item[(1)]  the pair $C^+,C^-$ is u-placed;
\item[(2)]  a continuous function $f:C^+\cup C^-\to \R$ is u.c. whenever $f|_{C^+}$ and $f|_{C^-}$ are u.c.
\item[(3)] same as (2) with $\R$ replaced by a general uniform space $(M,{\mathcal V})$
\end{enumerate}
\end{theorem}

The next theorem extends the defining property of straight spaces to arbitrary finite products. 

\begin{theorem}\label{SvsWS} {\rm \cite{BDP3}}
If a metric space $X$ is {straight} and $X$ can be written as a union of finitely many closed sets $C_1, \ldots, C_n$ it follows that $f \in C(X)$ is u.c. if and only if each restriction $f|_{C_k}$ ($k=1,2, \ldots, n$) of $f$ is u.c.
\end{theorem}



\begin{defn}\label{SCS_ULC_w} Let $X$ be a metric space. We say that  $X$ is $WULC$,       if for every pair of sequences $x_n$, $y_n$ in $X$ with
       $d(x_n,y_n)\to 0$ and such that the set $\{x_n\}$ is closed and discrete
\footnote{so that also the set $\{y_n\}$ is closed and discrete.} there exist a $n_0\in \N$ and connected sets
       $I(x_n,y_n)$ containing $x_n$ and $y_n$ for every $n\geq n_0$ in such       a way that the $\mbox{diam}I(x_n,y_n)\to 0$.
   \end{defn}

\begin{prop} Every $WULC$ space is straight. \end{prop}

\begin{proof}
 Assume $X$ is the union of finitely many closed sets $F_1, \ldots , F_m$ and the restriction of  a function $f \in C(X)$ to each of the closed sets
$F_k$ is u.c.  We have to check that $f$ is u.c. Pick $\eps>0$ and assume that $$|f(x_n)-f(y_n)|\geq 2.\eps\eqno(*)$$ for some $x_n$,
$y_n$ such that $d(x_n,y_n)\to 0$. It is clear, that the  sequence $x_n$ cannot have an accumulation point $x$
in $X$, since then some subsequence $x_{n_k}\to x$ and also $y_{n_k}\to x$. Now the continuity of $f$ would imply 
$|f(x_{n_k})-f(x)|\to 0$ and $|f(y_{n_k})-f(x)|\to 0$. Consequently, $|f(x_{n_k})-f(y_{n_k})|\to 0$
contrary to (*).  Therefore, the double sequence $x_n$, $y_n$ satisfies the condition (a) of Definition \ref{SCS_ULC_w}. Therefore, for large enough 
 $n$ we have a connected set $I_n$ containing $x_n, y_n$ such that $\mbox{diam}I_n <\delta$. We can choose $\delta>0$ such that also 
$|f(x)-f(y)|<\eps/m$ whenever $x, y$ belong to the same closed set $F_k$ and $d(x,y)<\delta$. 
Note that $f(I_n)$ is an interval with length $\geq 2\cdot \eps$ covered by $m$ subsets $f(I_n\cap F_k)$, $k=1,\ldots, m$ each with diameter $\leq \eps/m$. This leads to a
contradiction since an interval of length $\geq 2\cdot \eps$ cannot be covered by $m$ sets of diameter $\leq \eps/m$. \end{proof}

As a corollary we obtain that ULC spaces are straight.

\begin{theorem}\label{thm2.5}
 Let $(X, d)$ be locally connected. Then $(X, d)$ is straight if and only if it is uniformly locally connected. 
\end{theorem}

\begin{theorem}\label{thm2.6}
 Let $(X, d)$ be a totally disconnected metric space. Then $X$ is straight if and only if $X$ is UC. 
\end{theorem}

\subsection{Stability properties of straight spaces}

The next theorem shows that straightness spectacularly fails to be preserved under taking closed spaces. 

\begin{theorem}\label{thm2}
For every metric space $X$ with $\mbox{Ind} X=0$ the following are equivalent:
\begin{enumerate}
\item[(1)] $X$ is UC;
\item[(2)] every closed subspace of $X$ is straight;
\item[(3)]  whenever $X$ can be written as a union of a locally finite family  $\{C_i\}_{i\in I}$ of closed sets we have that $f \in C(X)$ is u.c. if and only if each restriction $f|_{C_i}$ of $f$, $i\in I$, is u.c.
\end{enumerate}
\end{theorem}

The following notion is relevant for the description of the dense straight subspaces. 

\begin{defn}\label{def tight}(\cite{BDP5}) An extension $X\subseteq Y$ of topological spaces is called {\bf tight} if for
every closed binary cover $X=F^+\cup F^-$ one has
\begin{equation}\label{def_ttt}
     \overline{F^+}^Y\cap \overline{F^-}^Y=\overline{F^+\cap F^-}^Y.
\end{equation}
\end{defn}

With this notion one can characterize straightness of extensions. 

\begin{theorem}\label{str-dense}(\cite{BDP5}) Let $X,\ Y$ be metric spaces, $X\subseteq Y$ and let $X$ be dense in $Y$. Then
$X$ is straight if and only if $Y$ is straight and the extension $X\subseteq Y$
is \ttt. \end{theorem}




\subsection{Products of straight spaces}\label{Prod_str}

Here we discuss preservation of straightness under products. 

Nishijima and Yamada \cite{Y} proved the following 

\begin{theorem} {\rm (\cite{Y})}\label{yam} Let $X$ be a straight space. Then $X\times K$ is straight for each compact space $K$ if and only if $X\times (\omega +1)$ is straight.
\end{theorem}

The next lemma easily follows from the definitions. 

\begin{lemma}\label{prod_ULC} {\rm (\cite{BDP6})} A product $X\times Y$ is ULC if and only if both $X$ and $Y$ are ULC.
\end{lemma}

The next proposition, proved in \cite{BDP6}, plays a crucial role in the proof of Theorem \ref{thmPROD}: 

\begin{prop}\label{Thm_loc_conn} 
If $X \times Y$ is straight, then $X$ is ULC or $Y$ is precompact. \end{prop}

\begin{theorem}\label{thmPROD} {\rm (\cite{BDP6})} The product $X\times Y$ of two metric spaces is straight if and only if both $X$ and $Y$ are straight and one of the following conditions holds:
\begin{itemize}
  \item[(a)]  both $X$ and $Y$ are precompact;
  \item[(b)] both $X$ and $Y$ are ULC;
  \item[(c)] one of the spaces is both precompact and ULC.    
\end{itemize}
\end{theorem}

It turns out that the straightness of an infinite product of ULC spaces is related to connectedness: 

\begin{theorem}\label{thmPROD1} {\rm (\cite{BDP6})} Let $X_n$ be a ULC space for each $n\in \N$ and $X = \Pi_n X_n$. 
\begin{itemize}
  \item[(a)] $X$ is ULC if and only if all but finitely many $X_n$ are connected.
  \item[(b)] The following are equivalent:
\begin{itemize}
    \item[(b$_1$)] $X$ is straight.
    \item[(b$_2$)] either $X$ is ULC or each $X_n$ is precompact.
\end{itemize}
\end{itemize}
\end{theorem}

This theorem completely settles the case of infinite powers of ULC space:

\begin{cor}\label{coroPROD}  Let $X$ be ULC. Then 
\begin{itemize}
   \item[(a)] $X^\omega $ is ULC if and only if $X$ is connected; 
   \item[(b)] $X^\omega$ straight if and only if $X$ is either connected or precompact.
\end{itemize}
\end{cor}

The above results leave open the question about when  infinite products of precompact straight spaces are still straight
(see Questions \ref{Ques-Inf} and \ref{Ques-inf}).

\section{Questions}\label{Questions} 

Our first open problem is about the implication (1) in Diagram 2: 

\begin{problem}\label{UA>WUA}{\rm (\cite[Problem 1.4]{CD2})} Does $WUA$ imply $UA$ in $C(\R^n)$? What about $C(\R^2)$?
\end{problem}

\subsection{Questions on UA functions and UA spaces}

A general question is to characterize the $UA$ and $WUA$ spaces and functions. We list below more specific questions {\rm (\cite{BD1})}.

\begin{enumerate}
\item Characterize the $UA$ functions $f\colon\R^2 \to \R$.

\item Characterize the $UA$ subsets of $\R$.

\item Characterize the topological spaces which admit a $UA$ uniformity, and those which are $UA$ under every uniformity
compatible with their topology. Does the latter class of spaces also include the $UC$ spaces?

\item Do $WUA$ and $UA$ coincide for connected spaces?

\item Suppose that a uniform space $X$ has a dense $UA$ subspace. Does it follow that $X$ is $UA$? (This fails for $WUA$ according to Example 7.7 from 
[BeDi1].)

\item Let $X$ be the pushout of two $WUA$ spaces over a single point. Is $X$ \ $WUA$? (This holds for $UA$ by Theorem 11.1  from 
[BeDi1].)

\item Suppose that every pseudo-monotone function $f \in C(X)$ is $UA$. Is then $X$ a $UA$  space?

\item Define $2$-$UA$ similarly as $UA$ but with the set $K$ of cardinality at  most $2$. Is then $2$-$UA$ equivalent to $UA$? 
\end{enumerate}


\subsection{Questions on thin spaces}

The next question is related to Theorem \ref{thmThin1}: 

\begin{question}
Is it true that a complete thin {\em uniform} space has $CSP$? What about a complete $UA$ uniform space? \end{question}

Our next  question is about how much one needs the fact that {\bf uniform} quasi  components are  connected.

\begin{question}
Is it true that every complete metric space $X$ such that every closed subspace of $X$ has connected quasi components necessarily  has CSP ?
\end{question}

\begin{question}\label{Ques_Thin}
Is it true that every (complete) metric thin separable space is $UA$?
\end{question}

\subsection{Questions on straight spaces}

Theorem \ref{str-dense} gives a criterion for straightness of a dense subspace $Y$ of a straight space $X$ in terms of properties of the  embedding $Y\hookrightarrow X$
(namely, when $X$ is a tight extension of $Y$).  The analogue of this question for {\em closed} subspaces is somewhat unsatisfactory. 
We saw that uniform retracts, clopen subspaces, as well as direct summands, of straight spaces are always straight ([BDP2]). 
On the other hand, closed subspaces even of ULC spaces may fail to be straight (see [BDP2]). 
Another instance when a closed subspace of a straight space fails to be straight is given by the following fact proved in \cite{BDP3}:
 the spaces $X$ in which {\em every} closed subspace is straight are precisely the UC spaces \cite{BDP3}.  Hence every straight space that is not UC has closed non-straight subspaces. This motivates the following general

\begin{problem}\label{QuesProd} {\em Find a sufficient condition ensuring that a closed subspace $Y$ of a straight space $X$ is straight. }
\end{problem}

\begin{question} Generalize the results on straight spaces from the category of metric spaces to the category of uniform spaces. \end{question}

The results from \S \ref{Prod_str} describe when infinite products of ULC spaces are again ULC or straight. The case of 
precompact spaces is still open, so we start with the following still unsolved

\begin{question}\label{Ques-inf} Let $X$ be a precompact straight space. Is the infinite power $X^\omega$ necessarily straight?
\end{question}

More generally:

\begin{question}\label{Ques-Inf} Let $X_n$ be a precompact straight space  for every $n\in \N$. Is the infinite product $\prod_n X_n$ necessarily straight?
\end{question}

It is easy to see that a positive answer to this question is equivalent to a positive answer to item (b) of the following general question:
 (i.e., the version of Theorem 30 for products of {\em precompact} spaces):

\begin{question}\label{ppstr} Let the metric space $Y_i$ be a tight extension of $X_i$  for each $i\in \N$. 
\begin{itemize}
\item[(a)] Is $\Pi_i Y_i$  a tight extension of $\Pi_i X_i$.
\item[(b)]  What about {\em precompact} metric spaces $Y_i$? 
\end{itemize} 
\end{question}

\bigskip
\section*{Acknowledgements }
This research was supported by the grant MTM2009-14409-C02-01
and Slovenian Research Agency Grants P1-0292-0101, J1-2057-0101 and J1-9643-0101.

\label{\mylastpage}

\end{document}